\documentclass[12pt]{amsart}
\usepackage[cp850]{inputenc}
\usepackage{graphics}

\newtheorem{theorem}{Theorem}

\newtheorem{corollary}[theorem]{Corollary}

\theoremstyle{definition}

\theoremstyle{remark}

\usepackage{amssymb,amsmath}

\newcommand{\g}[2]{\mbox{$\langle #1 ,#2 \rangle$}}
\newcommand{\s}[1]{\mbox{$\mathbb{S}^{#1}$}}

\newcommand{\R}[1]{\mbox{$\mathbb{R}^{#1}$}}

\renewcommand{\S}{\mbox{$\Sigma$}}
\newcommand{\M}{\mbox{$\mathcal{M}(\kappa,\tau)$}}
\newcommand{\Mi}{\mbox{$M$}}
\newcommand{\B}{\mbox{$B$}}

\newcommand{\nablabar}{\mbox{$\overline{\nabla}$}}

\newcommand{\fm}{\mbox{$\mathcal{C}^\infty(\S)$}}

\newcommand{\imm}{\mbox{$\psi:\S\fle\M$}}
\newcommand{\x}{\mbox{$\psi:\S^2\fle\Mi^3$}}

\renewcommand{\H}{\mbox{$\mathcal{H}$}}
\newcommand{\HE}{\mbox{$\mathcal{H}E$}}
\newcommand{\V}{\mbox{$\mathcal{V}$}}
\newcommand{\VE}{\mbox{$\mathcal{V}E$}}
\newcommand{\A}{\mbox{$\mathcal{A}$}}
\newcommand{\T}{\mbox{$\mathcal{T}$}}

\newcommand{\fle}{\mbox{$\rightarrow$}}
\newcommand{\rf}[1]{\mbox{(\ref{#1})}}

\def\area{\mathop\mathrm{Area(\S)}\nolimits}

\def\Kbar{\mathop{\bar{K}}\nolimits}
\def\KbarS{\mathop{\bar{K}_\Sigma}\nolimits}
\def\RiemM{\mathop\mathrm{\overline{R}}\nolimits}
\def\Ric{\mathop\mathrm{\overline{Ric}}\nolimits}
\def\RicN{\mathop\mathrm{\overline{Ric}(\textit{N},\textit{N})}\nolimits}

\def\xfrak{\mathfrak{X}}
\def\E3{\mbox{$\mathbb{E}^3(\kappa,\tau)$}}

\def\integ{\int_\Sigma}

\begin{document}

\title[First stability eigenvalue characterization of Hopf tori]{First stability eigenvalue characterization of CMC Hopf tori into Riemannian Killing submersions}



\author{Miguel A. Mero\~no}
\address{Departamento de Matem\'{a}ticas, Universidad de Murcia, E-30100 Espinardo, Murcia, Spain}
\email{mamb@um.es}
\thanks{This work was partially supported by MECD (Ministerio de Educaci\'on, Cultura y Deporte) Grant  no FPU12/02252, MINECO (Ministerio de Econom\'{\i}a y Competitividad) and FEDER (Fondo Europeo de Desarrollo Regional) project MTM2012-34037 and Fundaci\'{o}n S\'{e}neca project 04540/GERM/06, Spain. This research is a result of the activity developed within the framework of the Programme in Support of Excellence Groups of the Regi\'{o}n de Murcia, Spain, by Fundaci\'{o}n S\'{e}neca, Regional Agency for Science and Technology (Regional Plan for Science and Technology 2007-2010).}

\author{Irene Ortiz}
\address{Departamento de Matem\'{a}ticas, Universidad de Murcia, E-30100 Espinardo, Murcia, Spain}
\email{irene.ortiz@um.es}





\begin{abstract}
We find out upper bounds for the first eigenvalue of the stability
operator for compact constant mean curvature orientable surfaces immersed in a
Riemannian Killing submersion. As a consequence, the strong stability of such surfaces is studied. We also characterize constant mean curvature Hopf tori as the only ones attaining the bound in certain cases.
\end{abstract}

\maketitle

\section{Introduction}

Let \x\ be an isometric immersion of a compact orientable surface in a  three dimensional oriented Riemannian manifold. Fixed a unit normal vector field $N$ globally
defined on \S, we will denote by $A$ the second fundamental form with respect to $N$ of the immersion and by $H$ its mean curvature.

Every smooth function $f\in\fm$ induces a normal variation
$\psi_t$ of the immersion $\psi$, with variational normal field $fN$ and first variation of the
area functional $\mathcal{A}(t)$ given by
$
\delta_f\mathcal{A}=\mathcal{A}'(0)=-2\integ fH.
$
As a consequence, minimal surfaces ($H=0$) are characterized as the critical points of the area
functional whereas constant mean curvature (CMC) surfaces are the critical points of the area functional
restricted to smooth functions $f$ such that $\integ f=0$, which means that the variation leaves constant the volume enclosed by the surface.

For such critical points, the stability of the corresponding variational problem is given by the
second variation of the area functional,
\[
\delta^2_f\mathcal{A}=\mathcal{A}''(0)=-\integ  fJf,
\]
with $Jf=\Delta f+\left(|A|^2+\RicN\right)f$, where $\Delta$ stands
for the Laplacian operator on \S\ and $\Ric$ denotes the Ricci
curvature of \Mi. The surface \S\ is said to be {\em strongly stable} if
$\delta^2_f\mathcal{A}\geq 0$, for every $f\in\fm$. The operator
$J=\Delta+|A|^2+\RicN$, which is a Schr\"odinger operator, is known as the {\em Jacobi} or {\em stability operator} of
the surface. The spectrum of $J$
\[
\mathrm{Spec}(J)=\{ \lambda_1<\lambda_2<\lambda_3<\cdots \}
\]
consists of an unbounded increasing sequence of eigenvalues $\lambda_k$ with finite multiplicities. Moreover, the first eigenvalue is simple
(multiplicity one) and it satisfies the following min-max characterization
\begin{equation}
\label{minmax} \lambda_1=\min \left\{\frac{-\integ  fJf}{\integ  f^2}: \quad f\in\fm,f\neq0
\right\}.
\end{equation}

In terms of the spectrum, \S\ is strongly stable if and only if $\lambda_1\geq 0$.

Observe that with our criterion, a real number $\lambda$ is an eigenvalue of $J$ if and only if
$Jf+\lambda f=0$ for some smooth function $f\in\fm$, $f\neq0$.

In 1968, Simons \cite{Sim} found out an estimate for the first eigenvalue of $J$ on any compact
minimal hypersurface in the standard sphere. In particular, for minimal surfaces in the 3-sphere he
proved that $\lambda_1=-2$ if the surface is totally geodesic and $\lambda_1\leq -4$ otherwise.
Later on, Wu \cite{Wu} characterized the equality by showing that it holds only for the minimal
Clifford torus. In the last decade, Perdomo \cite{Pe} gave a new proof of this spectral characterization
by getting an interesting formula that relates the first eigenvalue $\lambda_1$, the genus of the
surface, the area and a simple invariant. Al\'ias, Barros and Brasil \cite{ABBr} extended
Wu and Perdomo's results to the case of CMC hypersurfaces in the standard
sphere, characterizing some CMC Clifford tori. Very recently, Chen and Wang \cite{CW} have just given optimal estimates for $\lambda_1$ for linear Weingarten hypersurfaces in the sphere characterizing the equality.

In \cite{AMO} the authors study the same problem in homogeneous Riemannian 3-manifolds, in particular in Berger spheres, finding out upper bounds for $\lambda_1$ for compact CMC surfaces immersed into such manifolds. They also get a characterization of CMC Hopf tori in certain Berger spheres and in the product $\s{2}\times\s{1}$.

Homogeneous Riemannian 3-manifolds are a special kind of manifolds belonging to a more general structure with many more interesting examples: Riemannian Killing submersions. They are Riemannian 3-manifolds which fiber over a Riemannian surface and whose fibers are the trajectories of a unit Killing vector field. The study of immersed surfaces into such manifolds is a subject of increasing interest (see \cite{EO}, \cite{Amelia}, \cite{RST} and references therein). In these manifolds, there also exist Hopf tori, which are obtained as the total lift of closed curves by means of the submersion (if it has compact fibers). These surfaces are flat and have constant mean curvature when the curve has constant curvature. They first appear in \cite{Pink} for the Hopf fibration, where they play a key role, and they have been studied in several works (e.g. \cite{Barros} and \cite{BFLM}).

In this paper, we extend the results in \cite{AMO} to Riemannian Killing submersions giving some estimates for $\lambda_1$ for CMC compact orientable surfaces immersed in such manifolds. As a consequence, the strong stability of such surfaces is studied. We also characterize CMC Hopf tori as the only ones attaining the upper bound in certain cases.

\section{Riemannian Killing submersions}

Let $\Mi^3$ be a three dimensional oriented Riemannian manifold and $\pi:\Mi\fle B$ a
Killing submersion over a surface $\B^2$, i.e. a Riemannian
submersion whose vertical unit vector field $\xi$ is a unit Killing
vector field on \Mi\ (hence the fibers are geodesics). In this situation \Mi\ is called a Riemannian Killing submersion. We remind that a vector field on \Mi\ is vertical if it is always tangent to fibers
and horizontal if it is always orthogonal to fibers. If $\nablabar$
stands for the Levi-Civita connection of \Mi, we have
\begin{equation}
\label{tau} \nablabar_E\xi=\tau(E\wedge\xi),
\end{equation}
for all vector fields $E$ on \Mi, where $\wedge$ is the vector
product in \Mi\ and $\tau:\Mi\fle\R{}$ is a smooth function called
the bundle curvature of $\pi$ (see Proposition 2.6 in \cite{EO} and
Lemma 2 in \cite{SV}). We will denote by $\kappa$ the Gaussian
curvature of \B\ and use for such a Riemannian Killing submersion the notation
$\Mi=\M$. In particular, there are some important cases depending on
the functions $\kappa$ and $\tau$. For instance, if both of
these functions are constant, then \M\ is a homogeneous Riemannian
3-manifold and when $\tau=0$, \M\ is a product $\B^2\times\R{}$ or $\B^2\times\s{1}$, where $\B^2$ is an arbitrary Riemannian surface.

We will denote by \V\ and \H\ the projections of the tangent spaces
of  \M\ onto the subspaces of vertical and horizontal vectors
respectively (so, every arbitrary vector field $E\in\xfrak(\M)$ can
be decomposed as $\HE+\VE$). Throughout this paper, the letters $X$ and
$Y$ stand for horizontal vector fields.

It is worth pointing out that there are two tensors which are
related to a Riemannian submersion $\pi$ (see \cite{O}). They appear
naturally and are defined by
\[
\T_E F=\H\nablabar_{\VE}(\V F)+\V\nablabar_{\VE}(\H F),
\]
and
\[
\A_E F=\V\nablabar_{\HE}(\H F)+\H\nablabar_{\HE}(\V F),
\]
for all vector fields $E,F\in\xfrak(\M)$. They have some properties
that can be seen in \cite{O}.

Let us now compute the sectional curvature $\Kbar(X,E)$ of any
tangent plane to \M, where $\{X,E\}$ is an orthonormal basis
spanning the plane. Without loss of generality we can assume that
$X$ is a horizontal vector field.

If we complete $\{X\}$ to a positively oriented local orthonormal
basis $\{X,Y,\xi\}$, where $Y$ is a horizontal vector field, taking
into account \rf{tau} we easily get
\begin{align}
\label{conexiones}
\nonumber \nablabar_X X &=\alpha Y, \quad & \nablabar_Y X&=-\beta Y-\tau \xi, \quad & \nablabar_\xi X&=-\delta Y,\\
\nonumber \nablabar_X Y&=-\alpha X+\tau\xi,\quad & \nablabar_Y Y&=\beta X, \quad & \nablabar_\xi Y&=\delta X,\\
\nablabar_X \xi&=-\tau Y, \quad & \nablabar_Y \xi&=\tau X, \quad &
\nablabar_\xi\xi&=0,
\end{align}
where $\alpha=\g{\nablabar_X
X}{Y}$, $\beta=\g{\nablabar_Y Y}{X}$ and $\delta=\g{\nablabar_\xi
Y}{X}$.

By decomposing $E=\g{E}{Y}Y+\g{E}{\xi}\xi$,  we have
\begin{eqnarray*}
\Kbar(X,E)&=&\g{\RiemM(X,E)X}{E} \\
&=&\g{E}{Y}^2\Kbar(X,Y)+\g{E}{\xi}^2\Kbar(X,\xi)+2\g{E}{Y}\g{E}{\xi}\g{\RiemM(X,Y)X}{\xi},
\end{eqnarray*}
and we know by Lemma 2.8 in \cite{EO} that $\Kbar(X,Y) = \kappa-3\tau^2$ and $\Kbar(X,\xi)=\tau^2$. Thus we obtain
\begin{equation}
\Kbar(X,E)=\g{E}{Y}^2(\kappa-3\tau^2)+\g{E}{\xi}^2\tau^2+2\g{E}{Y}\g{E}{\xi}\g{\RiemM(X,Y)X}{\xi}.
\end{equation}

From Theorem 2 in \cite{O} the curvature $\g{\RiemM(X,Y)X}{\xi}$ can be computed by using the tensors \T\ and \A\ as follows
\[
\g{\RiemM(X,Y)X}{\xi}=\g{(\nablabar_X\A)_X Y}{\xi}+\g{\A_X
Y}{\T_{\xi}X}\\
-\g{\A_{Y}X}{\T_{\xi}X}-\g{\A_X X}{\T_{\xi}Y}.
\]

Since \A\ has the alternation property for horizontal vector fields
and $\T_{\xi}X=\V\nablabar_{\xi}X=0$ by \rf{conexiones}, the above
expression reduces to
\[
\g{\RiemM(X,Y)X}{\xi}=\g{(\nablabar_X\A)_XY}{\xi}.
\]

On the other hand
\begin{eqnarray*}
(\nablabar_X \A)_X Y&=&\nablabar_X(\A_XY)-\A_{\nablabar_X
X}Y-\A_X(\nablabar_XY) \\
&=& \nablabar_X(\tau\xi) - \alpha\A_YY+\alpha \A_XX-\tau\A_X\xi \\
&=& X(\tau)\xi.
\end{eqnarray*}

Summing up and bearing in mind that $\g{E}{Y}^2+\g{E}{\xi}^2=1$ we find that the sectional curvature is given by the following formula
\begin{equation}
\label{curvsec0}
\Kbar(X,E)=\kappa-3\tau^2-\g{E}{\xi}^2(\kappa-4\tau^2)+2\g{E}{Y}\g{E}{\xi}X(\tau).
\end{equation}

As a consequence of this computation we can derive an expression
of the Ricci tensor of \M\ in the normal direction $N$ to
the plane generated by the basis $\{X,E\}$. We can write
$N=\g{N}{Y}Y+\g{N}{\xi}\xi$ and then we have
\begin{eqnarray*}
\RicN&=&\g{\RiemM(X,N)X}{N}+\g{\RiemM(Y,N)Y}{N}+\g{\RiemM(\xi,N)\xi}{N}\\
&=&\Kbar(X,N)+\g{N}{\xi}^2\g{\RiemM(Y,\xi)Y}{\xi}+\g{N}{Y}^2\g{\RiemM(\xi,Y)\xi}{Y}\\
&=&\Kbar(X,N)+\Kbar(Y,\xi).
\end{eqnarray*}

Again, we know that $\Kbar(Y,\xi)=\tau^2$ and $\Kbar(X,N)$ can be calculated from \rf{curvsec0}. Thus we obtain
\[
\RicN=\kappa-2\tau^2-\g{N}{\xi}^2(\kappa-4\tau^2)+2\g{N}{Y}\g{N}{\xi}X(\tau).
\]

By setting $\nu=\g{N}{\xi}$, we can write the above equation as
\begin{equation}
\label{Ric}
\RicN=\kappa-2\tau^2-\nu^2(\kappa-4\tau^2)+2\nu\sqrt{1-\nu^2}X(\tau)
\end{equation}
because of $\g{N}{Y}^2+\g{N}{\xi}^2=1$ and we can suppose that $\g{N}{Y}\geq0$ without loss of generality.

Finally, equation \rf{curvsec0} can be simplified as follows.  Let us put $\xi=\g{E}{\xi}E+\g{N}{\xi}N$ so that $\g{E}{Y}\g{E}{\xi}=\g{\g{E}{\xi}E}{Y}=-\g{N}{Y}\g{N}{\xi}=-\nu\sqrt{1-\nu^2}$  and \rf{curvsec0} is reduced to
\begin{equation}
\label{cursec}
\Kbar(X,E)=\tau^2+\nu^2(\kappa-4\tau^2)-2\nu\sqrt{1-\nu^2}X(\tau).
\end{equation}

\section{Surfaces into Riemannian Killing submersions}

Let \imm\ be a compact orientable surface with constant mean curvature $H$ immersed into \M\ and choose a first positive eigenfunction $\rho\in\fm$
of the stability operator. Thus $J\rho=-\lambda_1\rho$ or, equivalently,
\begin{equation}
\label{deltarho}
\Delta\rho=-\left(\lambda_1+|A|^2+\RicN\right)\rho.
\end{equation}

Extending Perdomo's ideas \cite[Section 3]{Pe} to our more general case, one can compute
\[
\Delta
\textrm{log}\rho=\rho^{-1}\Delta\rho-\rho^{-2}|\nabla\rho|^2=-\left(\lambda_1+|A|^2+\RicN\right)-\rho^{-2}|\nabla\rho|^2,
\]
and integrate on \S\ to find
\[
\alpha=\integ\rho^{-2}|\nabla\rho|^2=-\lambda_1\area-\integ\left(|A|^2+\RicN\right),
\]
where $\alpha\geq0$ defines a simple invariant that is independent of the choice of $\rho$ because
$\lambda_1$ is simple. In other words
\[
\lambda_1=-\frac{1}{\area}\left(\alpha+\integ\left(|A|^2+\RicN\right)\right).
\]

Now from the Gauss equation, we obtain a relation between the norm of the shape operator $|A|^2$,
the sectional curvature $\KbarS$ of the tangent plane to \S\ in \M, and the Gaussian
curvature $K$ of the surface as $|A|^2=2(2H^2+\KbarS-K)$ and, by the Gauss-Bonnet Theorem, the
above formula becomes
\begin{equation}
\label{lambda1}
\lambda_1=-4H^2-\frac{1}{\area}\left(\alpha+8\pi(g-1)+\integ\left(2\KbarS+\RicN\right)\right).
\end{equation}

Let us pay attention to some very special surfaces immersed into a
Riemannian Killing submersion \M. If $N$ stands for the Gauss map of the
surface, they appear when the square of the {\em angle function}
$\nu=\g{N}{\xi}$ attains its maximum or minimum at any point.

So, if $\nu^2\equiv1$ we have an {\em horizontal surface} which
means that the tangent plane contains only horizontal vectors at any
point. For example, if $\M$ is a product $\B^2\times\R{}$ or
$\B^2\times\s{1}$, where $\B^2$ is any Riemannian surface of
Gaussian curvature $\kappa$ (case when $\tau\equiv 0$), the
horizontal surfaces are the slices $\B\times\{t\}$. It is easy to see that $\tau=0$ over a horizontal surface and so they are totally geodesic.  Therefore, for these surfaces $\lambda_1=0$ because of $|A|^2=0$ and $\RicN=0$ which yields $J=\Delta$. In particular, horizontal surfaces are strongly stable.

On the other hand, when $\nu\equiv0$ the surface is the total lift $\pi^{-1}(\gamma)$ of some regular curve $\gamma$ in $\B$ and it is called the {\em Hopf cylinder} over $\gamma$ or, in particular, {\em Hopf torus} if the fibers and the curve are closed. These surfaces are flat and their mean curvature $H$ is $k_g/2$, where $k_g$ is the geodesic curvature of $\gamma$ (see \cite{EO}), so that they have constant mean curvature when the curve has constant curvature, and in particular, they are minimal when lifting geodesic curves.

For a CMC Hopf torus $\pi^{-1}(\gamma)$, with $\kappa$ constant over $\gamma$, we have $|A|^2=4H^2+2\tau^2$ and $\RicN=\kappa-2\tau^2$, hence its Jacobi operator is $J=\Delta+4H^2+\kappa$, and we immediately obtain $\lambda_1=-4H^2-\kappa$. Let us observe that if $\kappa$ is a positive constant over $\gamma$, then the Hopf torus $\pi^{-1}(\gamma)$ is not strongly stable, and when $\kappa$ is non-positive the strong stability of  the torus depends on its mean curvature.

As a direct application of the above computations, we can get upper bounds for $\lambda_1$ for
compact CMC surfaces immersed into a Riemannian Killing submersion and characterize the surfaces which attain the upper bound: these are just the special ones introduced before. Let us observe that the sign of $\kappa-4\tau^2$ plays an important role for homogeneous Riemannian 3-manifolds.  So, when $\kappa-4\tau^2=0$ the manifold corresponds to a quotient of the space forms $\s{3}$ or $\R{3}$, whose isometry group has dimension 6, and otherwise the manifold has isometry group of lower dimension. In this last case, we have the Berger spheres which experiment a very different behaviour according to the mentioned sign. Therefore it seems reasonable to distinguish between both cases in the general study of Riemannian Killing submersions.
\begin{theorem}
\label{Th+}
Let \M\ be a Riemannian Killing submersion with $\kappa-4\tau^2>0$ and $\S^2$
be a compact orientable surface of constant mean curvature $H$
immersed into \M. If $\lambda_1$ stands for the first eigenvalue
of its Jacobi operator, then
\begin{enumerate}
\item[(i)]
$\displaystyle{\lambda_1\leq-2H^2-\frac{1}{\area}\integ(2\tau^2-|\nabla\tau|)}$,
with equality if and only if  \S\ is a horizontal surface; and
\item[(ii)]
$\displaystyle{\lambda_1\leq-4H^2-\frac{8\pi(g-1)}{\area}-\frac{1}{\area}\integ(\kappa-|\nabla\tau|)}$,
with equality if and only if  \S\ is a Hopf torus over a constant
curvature closed curve, and both $\kappa$ and $\tau$ are constant over
\S.
\end{enumerate}
\end{theorem}
\begin{proof}
(i) Using the constant function $f=1$ as a test function in \rf{minmax} to estimate $\lambda_1$, one
easily gets that
\begin{eqnarray*}
\label{lambda1.1}
\nonumber \lambda_1&\leq&-2H^2-\frac{1}{\area}\integ\RicN-\frac{1}{\area}\integ|\phi|^2\\
&\leq& -2H^2-\frac{1}{\area}\integ\RicN,
\end{eqnarray*}
where $\phi$ is the total umbilicity tensor of \S, defined as $\phi=A-HI$. Observe that $|\phi|^2=|A|^2-2H^2\geq0$, with equality if and only if \S\ is totally umbilical.

From \rf{Ric}, taking into account that $|2\nu\sqrt{1-\nu^2}|\leq
1$, $\nu^2\leq 1$ and
$|X(\tau)|=|\g{X}{\nabla\tau}|\leq|\nabla\tau|$, we have
$\RicN\geq2\tau^2-|\nabla\tau|$ and obtain the inequality. If the
equality holds, then $\RicN=2\tau^2-|\nabla\tau|$ and again by
\rf{Ric}, we get  $\nu^2$=1 and so \S\ is a horizontal surface.
Conversely, if \S\ is a horizontal surface, then we know that it is
totally geodesic, $\tau=0$ over the surface and $\lambda_1=0$.
Therefore the equality holds.

(ii) Let us suppose that the tangent plane to \S\ is generated by a local orthonormal basis $\{X,E\}$, where $X$ is horizontal. So, using \rf{Ric} and \rf{cursec}, the integrand in \rf{lambda1} is
\begin{equation}
\label{2K+Ric}
2\KbarS+\RicN =
\kappa+\nu^2(\kappa-4\tau^2)-2\nu\sqrt{1-\nu^2}X(\tau)
\end{equation}
that we can estimate as
\begin{equation}
\label{integrando}
2\KbarS+\RicN\geq\kappa-|\nabla\tau|.
\end{equation}

By this inequality and the fact that $\alpha\geq0$, \rf{lambda1}
directly yields to the announced estimate.

If the equality holds, \rf{integrando} does as well and so the angle function $\nu$ has to be identically null which means that \S\ is a Hopf torus over a constant curvature closed curve. Moreover, $\alpha=0$ which implies $\rho$ is constant and then from \rf{deltarho}
\begin{equation}
\label{eq69}
\lambda_1+|A|^2+\RicN=0.
\end{equation}

By \rf{Ric}, $\RicN=\kappa-2\tau^2$ and the Gauss equation yields $|A|^2=4H^2+2\tau^2$. Thus, taking into account that $g=1$, \rf{eq69} becomes
\[
\kappa=\frac{1}{\area}\integ(\kappa-|\nabla\tau|),
\]
which shows us that both $\kappa$ and $\tau$ are constant over \S.
Reciprocally, if \S\ is a Hopf torus over a constant curvature
closed curve and both $\kappa$ and $\tau$ are constant over $\S$, then $\nabla\tau=0$,
$g=1$ and $\lambda_1=-4H^2-\kappa$, as we have seen before. So we conclude that the equality is
satisfied.
\end{proof}

Observe that for the existence of such tori in Theorem \ref{Th+}, the submersion must have compact fibers, which occurs in several significant cases (i.e. Berger spheres and other examples that we will see in the last section).

Let us mention some interesting consequences of this theorem related to the strong stability of a surface immersed into a Riemannian Killing submersion.

\begin{corollary}
Let \M\ be a Riemannian Killing submersion with $\kappa-4\tau^2>0$. If  $\S^2$
is a strongly stable compact orientable surface of constant mean
curvature $H$ immersed into \M\, then
\begin{enumerate}
\item[(i)]
$\displaystyle{H^2\leq\frac{1}{\area}\integ\left(\frac{|\nabla\tau|}{2}-\tau^2\right)}$,
with equality if and only if  \S\ is a horizontal surface; and
\item[(ii)]
$\displaystyle{H^2<\frac{2\pi(1-g)}{\area}+\frac{1}{4\area}\integ(|\nabla\tau|-\kappa)}$.
\end{enumerate}
\end{corollary}
\begin{proof}
As \S\ is strongly stable we know that
$\lambda_1\geq0$, so inequalities are direct consequences of Theorem
\ref{Th+}. In the second case, equality is not possible because if it holds,
then $\lambda_1=0$, and again by (ii) of Theorem  \ref{Th+} we obtain that
$\S$ is a Hopf torus and $\kappa$ is a positive constant over $\S$.
Hence, we get a contradiction because $\lambda_1<0$ for such tori, as we previously saw.
\end{proof}

So if $|\nabla\tau|\leq2\tau^2$, the only strongly stable CMC compact orientable surface immersed into \M\ are the horizontal ones with $|\nabla\tau|=2\tau^2$ over \S\ and when $|\nabla\tau|\leq\kappa$, there does not exist a strongly stable CMC compact orientable surface immersed into \M\ with $g\geq 1$.

It is worth pointing out the special case when the bundle curvature is constant, so we can set the following.

\begin{corollary}
Let \M\ be a Riemannian Killing submersion with constant bundle curvature $\tau$ such that $\kappa-4\tau^2>0$ and $\S^2$ be a compact orientable surface of constant mean curvature $H$ immersed into \M. If $\lambda_1$ stands for the first eigenvalue of its Jacobi operator, then
\begin{enumerate}
\item[(i)]
$\displaystyle{\lambda_1\leq-2(H^2+\tau^2)}$, with equality if and
only if  \S\ is a horizontal surface; and
\item[(ii)]
$\displaystyle{\lambda_1\leq-4H^2-\frac{8\pi(g-1)}{\area}-\frac{1}{\area}\integ\kappa}$,
with equality if and only if  \S\ is a Hopf torus over a constant curvature closed curve $\gamma$ and $\kappa$ is constant over $\gamma$.
\end{enumerate}
\end{corollary}

\begin{corollary}
Let \M\ be a Riemannian Killing submersion with constant bundle curvature $\tau$ such that $\kappa-4\tau^2>0$. The only strongly stable CMC compact orientable surface immersed into \M\ are the horizontal ones.
\end{corollary}

Now, we are going to analyze the other case $\kappa-4\tau^2<0$.

\begin{theorem}
\label{Th-}
Let \M\ be a Riemannian Killing submersion with $\kappa-4\tau^2<0$ and $\S^2$ be a compact orientable surface of constant mean curvature $H$ immersed into \M. If $\lambda_1$ stands for the first eigenvalue of its Jacobi operator, then
\begin{enumerate}
\item[(i)]
$\displaystyle{\lambda_1\leq-2H^2-\frac{1}{\area}\integ(\kappa-2\tau^2-|\nabla\tau|)}$,
with equality if and only if  \S\ is a Hopf torus over a closed geodesic $\gamma$,
$\tau=0$ over \S\ and $\kappa$ is constant over $\gamma$; and

\item[(ii)]
$\displaystyle{\lambda_1\leq-4H^2-\frac{8\pi(g-1)}{\area}-\frac{1}{\area}\integ(2\kappa-4\tau^2-|\nabla\tau|)}$,
with equality if and only if  \S\ is a horizontal surface with
Gaussian curvature $K=\kappa$.
\end{enumerate}
\end{theorem}
\begin{proof}
(i) Using the same reasoning as in the first part of
Theorem \ref{Th+}, we get the inequality because from \rf{Ric} we know
$\RicN\geq\kappa-2\tau^2-|\nabla\tau|$. Moreover, if the equality
holds, then $\phi=0$ and $\RicN=\kappa-2\tau^2-|\nabla\tau|$, so $\S$
is totally umbilical and $\nu\equiv0$, which means that $\S$ is a
totally umbilical Hopf torus over a constant curvature closed curve. By the total umbilicity we know $|A|^2=2H^2=k_g^2/2$, and the Gauss equation reduces to
$k_g^2/2=k_g^2+2\tau^2$. So we conclude $k_g=0$ and $\tau=0$
over $\S$. Now from \rf{lambda1} we get $\alpha=0$ which implies, as we saw in the proof of Theorem \ref{Th+}, that
$\kappa$ is constant over $\S$. Conversely, if $\S$ is a
Hopf torus over a closed geodesic $\gamma$ and
$\kappa$ is constant over $\gamma$, we know that $\lambda_1=-\kappa$. Since $\tau=0$ over \S,
the equality holds.

(ii) The proof of the inequality is analogous with (ii) in Theorem \ref{Th+}, but in this case from \rf{2K+Ric} we have
\[
2\KbarS+\RicN\geq2\kappa-4\tau^2-|\nabla\tau|.
\]

Now, if we suppose that equality holds, then
$2\KbarS+\RicN=2\kappa-4\tau^2-|\nabla\tau|$ which implies
$\nu^2\equiv1$, that is, $\S$ is a horizontal surface. As we have
already observed, $\S$ is totally geodesic and $\tau=0$ over $\S$,
so using \rf{cursec} we follow that $\KbarS=\kappa$. Moreover,
the Gauss equation simplifies to $\KbarS=K$, therefore $K=\kappa$.
Reciprocally, if \S\ is a horizontal surface, then $\lambda_1=0$. Since the Gaussian curvature
of \S\ is $K=\kappa$, \S\ is totally geodesic and
$\tau=0$ over \S, the right hand of the inequality is
\[
-\frac{8\pi(g-1)}{\area}-\frac{1}{\area}\integ 2K,
\]
which is zero by the Gauss-Bonnet Theorem, so the equality holds.
\end{proof}

\begin{corollary}
Let \M\ be a Riemannian Killing submersion with $\kappa-4\tau^2<0$. If  $\S^2$
is a strongly stable compact orientable surface of constant mean
curvature $H$ immersed into \M\, then
\begin{enumerate}
\item[(i)]
$\displaystyle{H^2<\frac{1}{\area}\integ\left(\tau^2+\frac{|\nabla\tau|}{2}-\frac{\kappa}{2}\right)}$; and
\item[(ii)]
$\displaystyle{H^2\leq\frac{2\pi(1-g)}{\area}+\frac{1}{\area}\integ\left(\tau^2+\frac{|\nabla\tau|}{4}-\frac{\kappa}{2}\right)}$,
with equality if and only if  \S\ is a horizontal surface with
Gaussian curvature $K=\kappa$.
\end{enumerate}
\end{corollary}
\begin{proof}
If $\S$ is strongly stable, then $\lambda_{1}\geq 0$ and from Theorem \ref{Th-}, we
derive both of these inequalities. Moreover, in the case (i)
equality does not satisfy because if it does, then $\lambda_1=0$
and from Theorem \ref{Th-}, $\S$ is a Hopf torus over a closed geodesic
$\gamma$, $\tau=0$ over \S\ and $\kappa$ is constant over $\gamma$,
so $\lambda_1=-\kappa>-4\tau^2=0$, which is a contradiction.
\end{proof}

\begin{corollary}
Let \M\ be a Riemannian Killing submersion with constant bundle curvature $\tau$ such that $\kappa-4\tau^2<0$ and $\S^2$ be a compact orientable surface of constant mean curvature $H$ immersed into \M. If $\lambda_1$ stands for the first eigenvalue of its Jacobi operator, then
\begin{enumerate}
\item[(i)]
$\displaystyle{\lambda_1\leq-2(H^2-\tau^2)-\frac{1}{\area}\integ\kappa}$, with equality if and only if  \S\ is a Hopf torus over a closed geodesic $\gamma$, $\tau=0$ and $\kappa$ is constant over $\gamma$; and

\item[(ii)]
$\displaystyle{\lambda_1\leq-4(H^2-\tau^2)-\frac{8\pi(g-1)}{\area}-\frac{2}{\area}\integ\kappa}$,
with equality if and only if  \S\ is a horizontal surface with
Gaussian curvature $K=\kappa$.
\end{enumerate}
\end{corollary}

As a consequence, if \M\ is a Riemannian Killing submersion with constant bundle curvature $\tau$ and Gaussian curvature $\kappa$ such that $0\leq\kappa<4\tau^2$, then any compact orientable surface of constant mean curvature $|H|>\tau$ immersed into \M\ can not be strongly stable, and if $|H|\leq\tau$ and the surface is strongly stable, then
\[
\area(\tau^2-H^2)\geq2\pi(g-1),
\]
with equality only if the surface is a flat horizontal one with $\kappa=0$.

\section{Some examples}

Let us see some interesting examples to which we can apply our results. As already mentioned, a Riemannian Killing submersion \M\ with both $\kappa$
and $\tau$ constant is a homogeneous Riemannian 3-manifold. For
these spaces, the consequences from the theorems are already
compiled in a previous paper (see \cite{AMO}). So, the results here
can be seen as a generalization of those.

Another remarkable case arises when the bundle curvature $\tau$ of
\M\ is null. Thus $\mathcal{M}(\kappa,0)$ is a product
$\B^2\times\R{}$ or $\B^2\times\s{1}$, where $\B^2$ is an arbitrary
Riemannian surface. When $\kappa$ is positive Theorem \ref{Th+}
establishes that for any compact orientable surface of constant mean
curvature $H$ immersed into $\mathcal{M}(\kappa,0)$ we have
$\lambda_1\leq-2H^2$ and the slices $\B\times\{t\}$ are the only
ones satisfying the equality (if \B\ is compact). Also observe
that if the surface is strongly stable, then it has to be minimal. On
the other hand, we know
\[
\lambda_1\leq-4H^2-\frac{8\pi(g-1)}{\area}-\frac{1}{\area}\integ\kappa
\]
and Hopf tori $\gamma\times\s{1}$ are the only ones that satisfy
equality, where $\gamma$ is any constant curvature closed curve such
that $\kappa$ is constant over $\gamma$. For the other case
$\kappa<0$ we have that
\[
\lambda_1\leq-2H^2-\frac{1}{\area}\integ\kappa,
\]
with equality only for minimal Hopf tori $\gamma\times\s{1}$ with
$\kappa$ constant over $\gamma$. And
\[
\lambda_1\leq-4H^2-\frac{8\pi(g-1)}{\area}-\frac{2}{\area}\integ\kappa,
\]
with equality only for slices $\B\times\{t\}$ with Gaussian
curvature $K=\kappa$.

And for the general case, where both $\kappa$ and $\tau$ are
non-constant, we can find a lot of examples in the work \cite{SV}.
Here the authors find out 3-dimensional spaces which locally admit a
doubly warped product metric and project over a certain surface as a
Riemannian Killing submersion. These spaces are very interesting because they
are characterized as the only ones carrying a unit Killing field
which admit totally geodesic surfaces different to a Hopf cylinder
or a horizontal surface, as they prove in the same work. The
following one is a particular example where we can apply the theory.

Let $M$ be the product $I\times\s{1}\times\s{1}$,  for some open
interval $I$ of $\R{}$ or $I=\s{1}$,  with the doubly warped product
metric $$ds^2=dx^2+\sin^2\theta(x)dy^2+\cos^2\theta(x)dz^2,$$ where
$\theta:I\fle(0,\pi/2)$ is any smooth function, and $B$ be the
warped product $I\times_f\s{1}$, with
$f(u)=\frac{1}{2}\sin(2\theta(u))$. It is a straightforward
computation to check that the application $\pi:M\fle B$ given by
$\pi(x,y,z)=(x,y-z)$ is a Riemannian Killing submersion with
$\xi=\partial_y+\partial_z$. The bundle curvature of $\pi$ is
$\tau=-\theta'$ and the Gaussian curvature of $B$ is
$\kappa=4(\theta')^2-2\cot(2\theta)\theta''$, and so
$\kappa-4\tau^2=-2\cot(2\theta)\theta''$. Thus, if $\theta<\pi/4$
and $\theta''<0$, Theorem \ref{Th+} asserts that for any compact
orientable surface of constant mean curvature $H$ immersed into $M$,
we have
\[
\displaystyle{\lambda_1\leq-2H^2-\frac{1}{\area}\integ(2(\theta')^2+\theta'')},
\]
and the equality holds for horizontal surfaces. And
\[
\displaystyle{\lambda_1\leq-4H^2-\frac{8\pi(g-1)}{\area}-\frac{1}{\area}\integ(4(\theta')^2+(1-2\cot(2\theta))\theta'')},
\]
and the equality holds for any Hopf torus
$\S=\pi^{-1}(\{u\}\times\s{1})$ over any parallel of  \B. An example
of such a function $\theta$ is given by $\theta(x)=1/2\tan^{-1}(x)$,
with $I=(0,+\infty)$. Observe that if we take
$\theta(x)=1/2\tan^{-1}(x)+\pi/4$, then $\kappa-4\tau^2<0$ and
Theorem \ref{Th-} applies. In this case, the equality cannot occur
in (i) because of $\tau\neq0$.

\bibliographystyle{elsarticle-num}

\end{document}